\documentclass[sigplan,screen,authorversion,nonacm]{acmart}
\ccsdesc[500]{Theory of computation~Logic and verification}
\ccsdesc[500]{Theory of computation~Constructive mathematics}
\ccsdesc[500]{Theory of computation~Type theory}
\ccsdesc[500]{Mathematics of computing~Algebraic topology}

\usepackage{textcomp}
\usepackage{xspace}
\usepackage[capitalise]{cleveref}
\usepackage{url}
\usepackage{multirow}
\usepackage{hhline}
\usepackage{pifont}
\usepackage{amsmath, amsfonts , amsthm , mathtools , bbold, float}
\usepackage{tikz}
\usetikzlibrary{arrows,matrix,decorations.pathmorphing,
  decorations.markings, calc, backgrounds}
\usepackage{mathpartir}
\usepackage{microtype}
\usepackage{hyperref}
\usepackage{url}
\usepackage{blindtext}
\usepackage{enumitem}
\usepackage{tikz-cd}
\usetikzlibrary{decorations.pathmorphing}
\DisableLigatures[-]{family=tt*}
\usepackage{microtype}

\AtEndPreamble{
\theoremstyle{acmdefinition}
\newtheorem{remark}[theorem]{Remark}}

\newcommand{\teletype}[1]{\ensuremath{\mathtt{#1}}}
\newcommand{\systemname}[1]{\teletype{\color{darkgray}#1}\xspace}

\newcommand{\Agda}{\systemname{Agda}}

\newcommand{\agdaCubical}{\systemname{agda/cubical}}
\newcommand{\Coq}{\systemname{Coq}}
\newcommand{\CoqEAL}{\systemname{CoqEAL}}
\newcommand{\CubicalAgda}{\systemname{Cubical} \systemname{Agda}}

\newcommand{\Lean}{\systemname{Lean}}

\definecolor{Revolutionary}{RGB}{232,70,68}

\usepackage{agda/latex/agda}
\newcommand{\anum}[1]{\AgdaNumber{#1}}

\newcommand{\func}[1]{\AgdaFunction{#1}}

\newcommand{\var}[1]{{\AgdaBound{#1}}}

\newcommand{\con}[1]{{\AgdaInductiveConstructor{\ensuremath{\mathsf{#1}}}}}
\usepackage{catchfilebetweentags}
\usepackage{newunicodechar}
\newunicodechar{≃}{\ensuremath{\simeq}}
\newunicodechar{≅}{\ensuremath{\cong}}
\newunicodechar{∈}{\ensuremath{\in}}
\newunicodechar{⁻}{\ensuremath{^{-}}}
\newunicodechar{ᴹ}{\ensuremath{_{M}}}
\newunicodechar{ᴺ}{\ensuremath{_{N}}}
\newunicodechar{≡}{\ensuremath{\equiv}}
\newunicodechar{λ}{\ensuremath{\lambda}}
\newunicodechar{⊎}{\ensuremath{\uplus}}
\newunicodechar{∷}{\ensuremath{::}}
\newunicodechar{ℓ}{\ensuremath{\ell}}
\newunicodechar{ᵢ}{\ensuremath{_i}}
\newunicodechar{⟨}{\ensuremath{\langle}}
\newunicodechar{⟩}{\ensuremath{\rangle}}
\newunicodechar{α}{\ensuremath{\alpha}}
\newunicodechar{β}{\ensuremath{\beta}}
\newunicodechar{θ}{\ensuremath{\theta}}
\newunicodechar{φ}{\ensuremath{\varphi}}
\newunicodechar{ψ}{\ensuremath{\psi}}
\newunicodechar{η}{\ensuremath{\eta}}
\newunicodechar{ε}{\ensuremath{\varepsilon}}
\newunicodechar{ι}{\ensuremath{\iota}}
\newunicodechar{Σ}{\ensuremath{\Sigma}}
\newunicodechar{σ}{\ensuremath{\sigma}}
\newunicodechar{∀}{\ensuremath{\forall}}
\newunicodechar{ℕ}{\ensuremath{\mathbb{N}}}
\newunicodechar{→}{\ensuremath{\to}}
\newunicodechar{⊎}{\ensuremath{\uplus}}
\newunicodechar{⋆}{\ensuremath{*}}
\newunicodechar{¬}{\ensuremath{\lnot}}
\newunicodechar{Θ}{\ensuremath{\theta}}
\newunicodechar{∧}{\ensuremath{\wedge}}
\newunicodechar{∨}{\ensuremath{\vee}}
\newunicodechar{∼}{\ensuremath{\sim}}
\newunicodechar{≢}{\ensuremath{\nequiv}}
\newunicodechar{Π}{\ensuremath{\Pi}}
\newunicodechar{ℤ}{\ensuremath{\mathbb{Z}}}
\newunicodechar{∥}{\ensuremath{\parallel}}
\newunicodechar{∣}{\ensuremath{\mid}}
\newunicodechar{ℚ}{\ensuremath{\mathbb{Q}}}
\newunicodechar{₊}{\ensuremath{_+}}
\newunicodechar{∗}{\ensuremath{\ast}}
\newunicodechar{₁}{\ensuremath{_1}}
\newunicodechar{₂}{\ensuremath{_2}}
\newunicodechar{⊥}{\ensuremath{\bot}}
\newunicodechar{∘}{\ensuremath{\circ}}
\newunicodechar{ρ}{\ensuremath{\rho}}
\newunicodechar{↦}{\ensuremath{\mapsto}}
\newunicodechar{₀}{\ensuremath{_0}}
\newunicodechar{∙}{\ensuremath{\boldsymbol{\cdot}}}
\newunicodechar{Ω}{\ensuremath{\Omega}}
\newunicodechar{§}{\ensuremath{\mathbb{S}}}
\newunicodechar{𝟙}{\ensuremath{\mathbb{1}}}
\newunicodechar{×}{\ensuremath{\times}}
\newunicodechar{⊕}{\ensuremath{\oplus}}
\newunicodechar{∃}{\ensuremath{\exists}}
\newunicodechar{♠}{\func{\ensuremath{\oplus}HIT}}
\newunicodechar{♥}{\ensuremath{B}}
\newunicodechar{⋆}{\func{\ensuremath{\star}}}
\newunicodechar{₋}{\func{\ensuremath{_-}}}
\newunicodechar{ℂ}{\func{\ensuremath{\mathbb{C}}}}
\newunicodechar{ℝ}{\func{\ensuremath{\mathbb{R}}}}
\newunicodechar{⋁}{\func{\ensuremath{\vee}}}
\newunicodechar{⊔}{\func{\ensuremath{\sqcup}}}
\newunicodechar{♢}{}

\newcommand{\N}{\mathbb{N}}
\newcommand{\Z}{\mathbb{Z}}

\newcommand{\R}{\mathbb{R}}

\newcommand{\Pco}[3]{#1[#2]/( #3 )}

\newcommand{\base}[2]{\con{base}\,#1\,#2}

\newcommand{\Vect}[2]{\func{Vec}\;#1\;#2}

\definecolor{dkblue}{rgb}{0,0.1,0.5}
\definecolor{lightblue}{rgb}{0,0.5,0.5}
\definecolor{dkgreen}{rgb}{0,0.6,0}
\definecolor{dkbrown}{rgb}{0.4,0,0}
\definecolor{dkviolet}{rgb}{0.6,0,0.8}

\newcommand{\mycomment}[3]{}

\newcommand{\trunc}[1]{\left\lVert#1\right\rVert}
\newcommand{\AgdaTrunc}[1]{\func{$\left\lVert\,#1\,\right\rVert$}}
\newcommand{\truncElem}[1]{\con{\mid}\,#1\,\con{\mid}}

\newcommand\ZMod{\ensuremath{\mathbb{Z}_2}}
\newcommand{\Klein}[0]{\ensuremath{\func{$K^2$}}}
\newcommand{\RP}[0]{\ensuremath{\func{$\mathbb{R}P^2$}}}
\newcommand{\RPW}[0]{\ensuremath{\func{$\mathbb{R}P^2 \,\vee \,\mathbb{S}^1$}}}
\newcommand{\SW}{\func{$\mathbb{S}^{\func{2}} \,\vee \,\mathbb{S}^{\textnormal{\func{4}}}$}}

\setcopyright{rightsretained}
\acmPrice{}
\acmDOI{10.1145/3573105.3575677}
\acmYear{2023}
\copyrightyear{2023}
\acmSubmissionID{poplws23cppmain-p27-p}
\acmISBN{979-8-4007-0026-2/23/01}
\acmConference[CPP '23]{Proceedings of the 12th ACM SIGPLAN International Conference on Certified Programs and Proofs}{January 16--17, 2023}{Boston, MA, USA}
\acmBooktitle{Proceedings of the 12th ACM SIGPLAN International Conference on Certified Programs and Proofs (CPP '23), January 16--17, 2023, Boston, MA, USA}

\begin{document}

\title{Computing Cohomology Rings in Cubical Agda}

\author{Thomas Lamiaux}
\email{thomas.lamiaux@ens-paris-saclay.fr}
\orcid{0000-0002-7318-5814}
\affiliation{%
  \institution{University Paris-Saclay and\\ ENS Paris-Saclay}
  \city{Gif-sur-Yvette}
  \country{France}
}

\author{Axel Ljungström}
\email{axel.ljungstrom@math.su.se}
\orcid{0000-0001-6946-0775}
\affiliation{%
  \institution{Department of Mathematics, Stockholm University}
  \city{Stockholm}
  \country{Sweden}
}

\author{Anders Mörtberg}
\email{anders.mortberg@math.su.se}
\orcid{0000-0001-9558-6080}
\affiliation{
  \institution{Department of Mathematics, Stockholm University}
  \city{Stockholm}
  \country{Sweden}
}

\begin{abstract}
  In Homotopy Type Theory, cohomology theories are studied
  synthetically using higher inductive types and univalence. This
  paper extends previous developments by providing the first fully
  mechanized definition of cohomology rings. These rings may be
  defined as direct sums of cohomology groups together with a
  multiplication induced by the cup product, and can in many cases be
  characterized as quotients of multivariate polynomial rings. To this
  end, we introduce appropriate definitions of direct sums and graded
  rings, which we then use to define both cohomology rings and
  multivariate polynomial rings. Using this, we compute the cohomology
  rings of some classical spaces, such as the spheres and the Klein
  bottle. The formalization is constructive so that it can be used to
  do concrete computations, and it relies on the \CubicalAgda system which
  natively supports higher inductive types and computational
  univalence.
\end{abstract}

\keywords{Synthetic Cohomology Theory, Cohomology Rings, Polynomials,
  Homotopy Type Theory}

\maketitle

\begin{acks}
  We are grateful for productive discussions with Evan Cavallo and Max
  Zeuner about the definition of graded rings as a HIT. We also thank
  Carl Åkerman Rydbeck for the formalization of the fact that
  \func{ListPoly} forms a commutative ring.

  This paper is based upon research supported by the Swedish Research
  Council (Vetenskapsrådet) under Grant No.~2019-04545. The research
  has also received funding from the Knut and Alice Wallenberg
  Foundation through the Foundation's program for mathematics.
\end{acks}

\section{Introduction}
\label{sec:introduction}

A fundamental idea in algebraic topology is that spaces can be
analyzed in terms of homotopy invariants---functorial assignments of
algebraic objects to spaces. Primary examples of such invariants
include homotopy groups, homology groups and, the concept of study in
this paper, cohomology groups and rings. Both homology and cohomology
groups are often much easier to compute than homotopy groups, making
them ubiquitous in modern pure mathematics as well as in
applied subjects, like topological data analysis \cite{TDA21}.

Intuitively, the cohomology groups $H^n(X,G)$ of a space $X$ relative
to an abelian group $G$, in particular $\Z$, characterize the connected
components of $X$ as well as its $(n+1)$-dimensional holes.
\cref{fig:circles} depicts the circle, $\mathbb{S}^1$, and two circles
that have been glued together in a point, i.e. the \emph{wedge sum} of two
circles, $\mathbb{S}^1 \vee \mathbb{S}^1$. The fact that these spaces have a
different number of holes is captured cohomologically by
$H^1(\mathbb{S}^1,\Z) \cong \Z$ and
$H^1(\mathbb{S}^1 \vee \mathbb{S}^1,\Z) \cong \Z \times \Z$, which
geometrically means that they have one, respectively two,
$2$-dimensional holes (i.e. empty interiors). As cohomology groups are
homotopy invariants, this then means that the spaces cannot be
continuously deformed into each other.

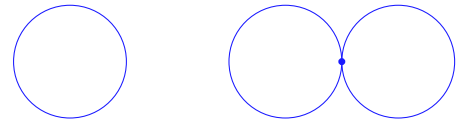
\begin{figure}[h]
  \centering
  \begin{mathpar}
  \begin{tikzpicture}[scale=0.5]
    \draw[color=blue!90](0.5,0) circle (1.5);
  \end{tikzpicture}
  \and
  \begin{tikzpicture}[scale=0.5]
    \draw[color=blue!90](5,0) circle (1.5);
    \draw[color=blue!90](8,0) circle (1.5);
    \filldraw[color=blue!90, very thick](6.5,0) circle (0.05);
  \end{tikzpicture}
  \end{mathpar}
  \caption{$\mathbb{S}^1$ and $\mathbb{S}^1 \vee \mathbb{S}^1$.}
  \label{fig:circles}
\end{figure}

More formally, one considers cohomology with coefficients in an
abelian group $G$ as a family of contravariant functors
$H^n(\_\hspace{.5mm},G) : \mathsf{Space} \to \mathsf{AbGroup}$
assigning to a space\footnote{We are intentionally vague about what a
  ``space'' is, but it can for example be a topological space or some
  more combinatorial presentation like simplicial or CW complexes, or
  as in HoTT/UF, a type.} $X$ its $n$th cohomology group with
$G$-coefficients. This assignment is assumed to satisfy the
Eilenberg-Steenrod axioms~\cite{eilenberg1952foundations} and be
stable under weak homotopy equivalence. This is a precise mathematical
way to say that the functors $H^n(\_\hspace{.5mm},G)$ identify spaces
of the same ``shape''. Such a family of functors is then said to
constitute a \emph{cohomology theory}.

Classically, homology groups have a definition very similar to that of
cohomology groups and can sometimes be easier to compute, but
cohomology groups (when taken with coefficients in a ring $R$) have an
important advantage: they can be equipped with a graded product
\begin{align*}
  \smile \, : H^n(X,R) \to H^m(X,R) \to H^{n+m}(X,R)
\end{align*}
This product, known as the \textit{cup product}, can then be used to construct a graded
ring $H^*(X,R)$, known as the \emph{cohomology ring} of $X$ with
coefficients in $R$. This gives an even more fine-grained invariant
than just the (co)homology groups, as it can help to distinguish spaces
for which all (co)homology groups agree. As an example, consider
\cref{fig:torusandmickeymouse} which depicts a torus and the ``Mickey
Mouse space'' consisting of a sphere glued together with two circles.

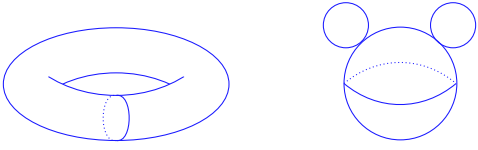
\begin{figure}[h!]
  \centering
  \begin{mathpar}
    \begin{tikzpicture}[scale=0.5]
      \useasboundingbox (-3,-1.5) rectangle (3,1.5);
      \draw[color=blue!90] (0,0) ellipse (3 and 1.5);
      \begin{scope}
        \clip (0,-1.8) ellipse (3 and 2.5);
        \draw[color=blue!90] (0,2.2) ellipse (3 and 2.5);
      \end{scope}
      \begin{scope}
        \clip (0,2.2) ellipse (3 and 2.5);
        \draw[color=blue!90] (0,-2.2) ellipse (3 and 2.5);
      \end{scope}
      \draw[color=blue!90] (0,-1.5) to[bend right=100] (0,-0.3);
      \draw[color=blue!90,densely dotted] (0,-1.5) to[bend left=100] (0,-0.3);
    \end{tikzpicture}
    \and
    \begin{tikzpicture}[scale=0.5]
      \draw[color=blue!90](-0.9,0) circle (1.5);

      \draw[color=blue!90](0.5,1.55) circle (0.6);
      \draw[color=blue!90](-2.35,1.55) circle (0.6);
      \draw[color=blue!90] (-2.4,0) to[bend right=40] (0.58,0);
      \draw[color=blue!90,densely dotted] (-2.4,0) to[bend left=40] (0.58,0);
    \end{tikzpicture}
  \end{mathpar}
  \caption{$\mathbb{T}^2$ and $\mathbb{S}^2 \vee \mathbb{S}^1 \vee \mathbb{S}^1$.}
  \label{fig:torusandmickeymouse}
\end{figure}

\noindent Let $X$ be either of the two spaces in the figure; one can
prove that $H^0(X,\Z) \cong \Z$ as they are connected,
$H^1(X,\Z) \cong \Z \times \Z$ as they each have two 2-dimensional
holes (the circle in the middle of the torus and the one going around
the interior, and the ``ears'' of Mickey), and $H^2(X,\Z) \cong \Z$ as
they have one 3-dimensional hole each (the interior of the torus and
Mickey's head). Furthermore, their higher cohomology groups are all
trivial as they do not have any further higher dimensional
cells. The cohomology groups are hence insufficient to tell the
spaces apart. However, one can show that the cup product on
$\mathbb{S}^2 \vee \mathbb{S}^1 \vee \mathbb{S}^1$ is trivial, while
for $\mathbb{T}^2$ it is not. So
$H^*(\mathbb{S}^2 \vee \mathbb{S}^1 \vee \mathbb{S}^1,\Z) \not \cong
H^*(\mathbb{T}^2,\Z)$, which again means that the spaces cannot be
continuously deformed into each other.

Providing a suitable framework for formalizing the graded ring
structure of the cohomology rings of a space is one of the main goals
of this paper. The formalization is carried out in
\CubicalAgda~\cite{cubicalagda2}, an extension of \Agda~\cite{Agda},
with native support for higher inductive types (HITs) and
computational univalence. This is a direct implementation of Homotopy
Type Theory and Univalent Foundations (HoTT/UF), where the idea that
equalities between terms of a type can be represented as paths in a
space is taken very literally. This makes it possible to develop
homotopy theory \emph{synthetically} as in the HoTT Book \cite{HoTT13}
and many classical results from homotopy theory have been formalized
this way.
Synthetic cohomology theory in Book HoTT was initially studied at the
IAS special year on HoTT/UF in 2012--2013~\cite{ShulmanBlog13} and has
since been used to develop the Eilenberg-Steenrod
axioms~\cite{CavalloMsc15}, cellular
cohomology~\cite{BuchholtzFavonia18}, Atiyah-Hirzebruch and Serre spectral sequences \cite{FlorisPhd}, and to prove that
$\pi_4(\mathbb{S}^3) \simeq \Z/2\Z$ \cite{Brunerie16}. By
developing these results cubically in \CubicalAgda, many
proofs can be substantially simplified, as illustrated for synthetic
homotopy theory by \citet{cubicalsynthetic} and for integral
cohomology theory by
Brunerie, Ljungström, and Mörtberg \cite{BLM}. This is made possible
by the fact that univalence computes, and more importantly, that all
computation rules for HITs hold strictly and not just up to paths as
in Book HoTT.

This present paper is a continuation of the work by \citet{BLM}, where
the integral cohomology groups $H^n(X,\Z)$ were developed and studied
in \CubicalAgda. There, the authors computed multiple $\Z$-cohomology
groups of various spaces represented as HITs, including the spheres,
torus, real/complex projective planes, and Klein bottle. They also
gave new synthetic constructions of the group structure on $H^n(X,\Z)$
and of the cup product, which led to substantially simplified proofs
compared to those of \citet{Brunerie16}. As all the proofs were
constructive, \CubicalAgda could be used to compute with these
cohomology operations. The present paper is also carefully written so
that all proofs are constructive and can be used to do concrete
computations.  Having the possibility of doing proofs simply by
computation is one of the most appealing aspects of developing
synthetic cohomology theory cubically. As this is not possible with
pen and paper proofs, or even with many formalized proofs in Book
HoTT, one often has to resort to doing long calculations by hand. If
proofs instead can be carried out using a computer, many of these long
calculations become obsolete.

\paragraph{Contributions}
The main contribution of the paper is the first fully formalized
synthetic definition of the cohomology ring $H^*(X,R)$ given a space
$X$ and ring $R$ in \cref{sec:cohomologyrings}. While previous authors
have discussed synthetic constructions of the cup
product~\citep{Brunerie16,Baumann18,BLM}, we lift it to a well-chosen
construction of $H^*(X,R)$. Furthermore, we compute some of these
rings for various $X$ and $R$, including the spheres, real and complex
planes, Klein bottle, and various wedges of these spaces. These
cohomology rings are all equivalent to some $R[X_1,\dots,X_n]/I$,
i.e. a (multivariate) polynomial ring modulo an ideal of relations.
These particular examples of spaces are chosen because they illustrate
different aspects of the computations as well as the need to sometimes
change the ring $R$ in order to distinguish spaces.

In order to construct $H^*(X,R)$, we give a general
formalization of graded rings (\cref{sec:formalizinggradedrings}),
which in turn requires us to formalize direct sums of $\N$-indexed
families of groups (\cref{subsec:directsums}). This is interesting on
its own, as special care has to be taken to remain constructive while
not imposing undesirable decidability assumptions. To facilitate
convenient formal proofs involving the direct sum, we give a new
definition of it using a HIT. As a byproduct, we get a new definition
of (multivariate) polynomial rings (\cref{sec:polynomialrings}) which
is well-suited both for programming and proving. As part of this, we
rely on the structure identity principle (SIP)---univalence for
\emph{structured} types---to transport proofs between different
representations of the direct sum
(\cref{subsubsec:sip}).

All results in the paper have been formalized in \CubicalAgda and is
part of the \agdaCubical library (available at
\url{https://github.com/agda/cubical/}). The code in the paper is
mainly literal \Agda code taken verbatim from the library, but we have
taken some liberties when typesetting, e.g. shortening notations and
omitting some universe levels. The connections between the paper and
formalization is summarized in a summary file which can be found on
\href{https://github.com/agda/cubical/blob/master/Cubical/Papers/CohomologyRings.agda}{GitHub}.
The summary file typechecks with \Agda's \texttt{--safe}
flag, which ensures that there are no admitted goals or postulates.

\section{Background}
\label{sec:background}

Here we briefly overview prior work and survey the fundamental
concepts underpinning the results of the paper.

\subsection{Homotopy Type Theory in \CubicalAgda}

\Agda is a dependently typed functional programming language, in which
each program may be interpreted as a proof in intensional type theory
\cite{MartinLof84bibliopolis}. For a brief overview of the syntax, see
\cref{table:agda}.
\begin{table}[h!]
  \centering
  \begin{tabular}{ |c | c| }
    \hline
    Concept & Syntax \\ [0.5ex]
    \hline \hline
    Function types & $A \to B$ \\
    Dependent function types & $(x : A) \to B$ \\
    Implicit dep. function types & $\{x : A\} \to B$ \\
    Function application, $f(x)$ & $f\,x$ \\
    Dependent pair types & $\func{Σ}\,A\,B$ and $\func{Σ[}\,x\,\func{∈}\,A\,\func{]}\,B\,x$ \\
    Universes (at level $\ell$) & $\func{Type}\,\var{ℓ}$ \\ [0.5ex]
    \hline
  \end{tabular}
  \caption{Overview of \Agda syntax}
  \label{table:agda}
\end{table}
With the addition of the univalence axiom of \citet{Voevodsky10cmu},
which says that equivalence of types can be promoted to
equalities/paths of types, \Agda may be used as a proof assistant for
various flavors of HoTT/UF. \CubicalAgda is an extension of \Agda for
a particular flavor of HoTT/UF -- the cubical type theory of
\citet{CCHM18} and \citet{CoquandHuberMortberg18}. In \CubicalAgda, unlike plain
\Agda, univalence is given computational content, allowing us to carry
out computations involving it. \CubicalAgda also has native support
for HITs. These types can, in particular, be used to define quotient
types and various spaces.

The major difference when working in \CubicalAgda compared to vanilla
\Agda or Book HoTT is that the primary identity type over a type $A$
is changed from Martin-Löf's inductive construction
\cite{MartinLof75itt} to a primitive \emph{path}-type. The
identification $x \,\func{$\equiv$}\,y$ is captured by
$\func{Path}\,A\,x\,y$, the type of functions $p : \func{I} \to A$,
where $\func{I}$ is a primitive interval type, restricting
definitionally to $x$ and $y$ at the endpoints \con{i0} and \con{i1}
of \func{I}. The interval also comes with join, meet and inversion
operations \func{\_∧\_}, \func{\_∨\_}, \func{∼{}\_}, endowing it with
the structure of a De Morgan algebra.

For an example of the cubical path type in action, consider the
following proof of function extensionality:
\ExecuteMetaData[agda/latex/Section2.tex]{funExt} Above, the term
$p\,x$ is interpreted as a path $\func{I} \to B$, restricting
definitionally to $f\,x$ at $\con{i0}$ and to $g\,x$ at
$\con{i1}$. Exploiting the $\eta$-rule for function types, the above
restricts definitionally to $f$ at $\con{i0}$ and to $g$ at
$\con{i1}$.

\CubicalAgda also has a dependent path type, \func{PathP}. Given a
line of types \var{A} : \func{I} $\to$ \func{Type} (which we
informally may think of as \var{A} \con{i0} \func{$\equiv$} \var{A}
\con{i1}) and points \var{x} : \var{A} \con{i0}, \var{y} : \var{A}
\con{i1}, the type \func{PathP} \var{A} \var{x} \var{y} expresses that
\var{x} and \var{y} may be identified relative to \var{A}. The regular
(homogeneous) path type $\func{\_$\equiv$\_}$ is, by definition, \func{PathP} ($\lambda$
\var{i} $\to$ \var{A}), i.e. the special case when the line of types
is constant. For an example of how \func{PathP} is used in
\CubicalAgda, consider the characterization of equality in
$\Sigma$-types: given (\var{x},\var{bx}), (\var{y},\var{by}) :
\func{Σ} \var{A} \var{B}, a path \var{p} : \var{x} \func{$\equiv$}
\var{y} and a dependent path \func{PathP} ($\lambda$ \var{i} $\to$
\var{B} (\var{p} \var{i})) \var{bx} \var{by}, we get a path of pairs
(\var{x},\var{bx}) \func{$\equiv$} (\var{y},\var{by}). In
\CubicalAgda:
\ExecuteMetaData[agda/latex/Section2.tex]{SigmaEq}

As previously mentioned, \CubicalAgda also supports HITs. In
particular, this allows us to construct various (homotopy types of)
topological spaces. For a simple example, consider the HIT describing
the circle, defined by a point \con{base} and the identification
\con{loop} : \con{base} \func{$\equiv$} \con{base}.
\ExecuteMetaData[agda/latex/Section2.tex]{S1}

HITs automatically come equipped with elimination principles. For
instance, in order to define a dependent function
$(x : \func{$\textnormal{S}^1$}) \to B\,x$, we need to provide a point
\var{b} : \var{B} \con{base} and, given $i : \func{I}$, a point $p_i$
reducing definitionally to $b$ when $i = \con{i0}$ and $i = \con{i1}$,
i.e. a dependent path
$\func{PathP}\,(\lambda \,i \to B\,(\con{loop}\,i)) \,b\,b$. In
\CubicalAgda, definitions involving HITs can also be written by
pattern-matching.  For a simple example, consider the definition of
the inversion map on $\func{{$\textnormal{S}^1$}}$.
\ExecuteMetaData[agda/latex/Section2.tex]{invS1}

Apart from topological spaces, HITs can also be used to capture
important type operations, such as truncations. For instance, the
propositional truncation is defined by:
\ExecuteMetaData[agda/latex/Section2.tex]{propTrunc}
This HIT takes a type $A$ and forces it to be a \emph{proposition}, or
$(-1)$-type: a type where any two points are identified up to a
path. This is an important construction for capturing logical
propositions in HoTT/UF. Indeed, we wish to think of logical
propositions as having at most one witness and not any higher
mathematical structure. For instance, we define existential
quantification as:
\begin{align*}
  \func{$\exists$[ }a \,\func{$\in$}\, A\,\func{]} (P\, a) = \AgdaTrunc{\Sigma\,{\color{black}{A \,P}}}_{\func{-1}}
\end{align*}
In this paper, we follow the HoTT Book terminology and say that $a$
\emph{merely} exists when it is existentially quantified. Also, note
that the propositional truncation in the definition is crucial. In
HoTT/UF $\func{Σ}\,A\,P$, is without the truncation interpreted as
the total space of $P$, which may be highly non-trivial.

We can also move one step up and define set
truncation:
\ExecuteMetaData[agda/latex/Section2.tex]{setTrunc}
The set truncation turns any type into a \emph{set}, or $0$-type: a
type where all path type are propositions. This notion allows us to
capture the idea of a classical set: a set in HoTT/UF is simply a
collection of points without any higher homotopical structure. We can
iterate this procedure and define $\func{$\trunc{\_}$}_n$ for any $n >
0$. The $n$-truncation $\func{$\trunc{{\color{black}{A}}}$}_n$ turns
$A$ into an $n$-type: a type over which all path types are
$(n-1)$-types. For the general definition of $\func{$\trunc{\_}$}_n$
one has to use a slightly different approach than we have for
$\func{$\trunc{\_}_{-1}$}$ and $\func{$\trunc{\_}_0$}$ based on the
hub-and-spoke construction~\cite[Chapter 7.3]{HoTT13}.

Later on, we will see that HITs also play an important role in the
construction of types quotiented by a relation (as opposed to
topological spaces). The idea is that we may add paths between points
to capture the relation we want to quotient by and then set truncate
to kill off higher homotopical structure arising from the imposed
relations.

Finally, we will need pointed types. A pointed type is simply a pair
$(A,a_0)$ where $A$ is a type and $a_0 : A$ is a chosen basepoint. In
\CubicalAgda, we define the universe of pointed types by:

\begin{center}
\ExecuteMetaData[agda/latex/Section2.tex]{Pointed}
\end{center}
Given a pointed type $A$, we write $\func{typ}\,A$ for the underlying
type and $\func{pt}\,A$ for the basepoint. We will, however, often
abuse notation and simply write $A$ for the pointed type $(A,a_0)$.

\subsection{The structure identity principle}
\label{subsection:SIP}

It is often the case that the types we are interested in come equipped
with some structure on it. For instance, a group consists of an underlying set
\var{G} with a group structure
on \var{G}:
\[
\func{GroupStr}\,G = \sum_{(1,\,\cdot\,,\,\_^{-1})\,:\,\func{RawGroupStr}\,G}\func{GroupAx}\,(G,\,1\,,\,\cdot\,,\,\_^{-1})
\]
Above, $\func{RawGroupStr}(G)$ expresses that $G$ has a raw group
structure consisting of a neutral element and the usual group
operations, whereas $\func{GroupAx}$ expresses that it satisfies the
group axioms. Note that $\func{GroupAx}$ is a proposition as \var{G}
is a set. Many other algebraic structures can be captured in this way,
e.g. monoids, rings, fields, etc. A reasonable question to ask in a
univalent setting is whether an equivalence of types can be promoted
to an equality of structured types, such as groups. The
\emph{Structure Identity Principle} (SIP) \citep[Section~9.8]{HoTT13} is
an informal principle which attempts to answer this: given two
structured types $(A,S_A)$ and $(B,S_B)$ and an equivalence of
underlying types $A \simeq B$ which is a homomorphisms with respect to
the structure in question, we get a path of structured types
$(A,S_A)\, \func{$\equiv$}\,(B,S_B)$. For instance, an isomorphism of
groups $G$ and $H$ induces a path $G \, \func{$\equiv$}\, H$. This
form of invariance up to isomorphism is one of the big strengths of
univalent mathematics as it lets programs and proofs be transported
between isomorphic structured types.

Perhaps more importantly, the SIP may also be used to transfer proofs
between (partially defined) structures. Suppose, for instance, that we
are given a group $G$ and type $H$ with just a raw group structure
$(1_H,\cdot_H,\_^{-1_H})$. While it may be difficult to show directly
that this raw group structure respects the group axioms, it may be
easier to show that the equivalence $\varphi : G \simeq H$ of types
preserves the raw structure from $G$ to $H$.\footnote{For groups it of
  course suffices to preserve only the $\cdot$ operation for $\varphi$
  to be a group homomorphism, but this is a quite special property of
  groups and for general structures the whole raw structure needs to
  be preserved.} The SIP then gives an induced group structure on $H$
which is path-equal, as groups, to $G$. Thus, the SIP allows us to
transport proofs between structured types in an efficient way.
This has been implemented in the \CubicalAgda standard library for
most algebraic structures using the cubical SIP of
Angiuli,~Cavallo,~Mörtberg~and~Zeuner~\cite{ACMZ21}.

\subsection{Cohomology theory in \CubicalAgda}

For the construction of cohomology in HoTT/UF, we use Brown
representability~\cite{BrownRepr} and define cohomology groups as
homotopy classes of maps into Eilenberg-MacLane spaces. Classically
this is provably equivalent to the singular cohomology of the spaces
we consider \cite{Hatcher2002}---but we take it as our definition.
Given an abelian group $G$, we
define the $n$th cohomology group of a type $X$ with $G$-coefficients
by:
\begin{align*}
  H^n(X,G) = \AgdaTrunc{{\color{black}{X}} \,{\color{black}{\to}} \,{\color{black}K(G,n)}}_{\func{0}}
\end{align*}
where $K(G,n)$ is the $n$th Eilenberg-MacLane space of $G$. Its
construction as a HIT is well-known in HoTT/UF and is due
to~\citet{LicataFinster14}. One can construct maps:
\begin{align*}
  \func{$+_k$} : K(G,n) &\to K(G,n) \to K(G,n) \\
  \func{$-_k$} : K(G,n) &\to K(G,n)
\end{align*}
and, for a ring $R$:
\begin{align*}
  \func{$\smile_k$} \, : K(R,n) \to K(R,m) \to K(R,n+m)
\end{align*}
These operations induce, by post-composition, operations on cohomology
groups, which we denote by $\func{$+_h$}$, $\func{$-_h$}$,
$\func{$\smile_h$}$, respectively.
Later on, we will see that these operations, by construction,
induce the ring structure on $H^*(X,G)$.

The cup product, with ring laws, was first introduced (in HoTT)
by~\citet{Brunerie16} for $\Z$-coefficients. These definitions and
results, in their entirety, have, however, never been formally
verified for Brunerie's original definition of the cup
product (although \emph{most} of them have by~\citet{Baumann18}).
Nevertheless, in~\citet{BLM}, the cup product was given a more concise, recursive
definition, which resulted in a complete formalization of the graded
ring axioms. The work in \citep{BLM} was only concerned with
$\Z$-coefficients however, but the authors emphasized that the results
are easily generalized to hold for arbitrary coefficients. The cup
product and ring axioms have since been formalized for arbitrary
coefficients and can be found in the \agdaCubical library. In
particular, it is shown that:
\begin{itemize}
\item The element $\func{$0_h$} : H^n(X,R)$ defined by
  \[
    \func{$0_h$} = \,\truncElem{\lambda \,x \to \func{$0_k$}}
    \]
    where $\func{$0_k$} : K(R,n)$ is the basepoint,
  right- and
    left-annihilates. That is, for $x : H^m(X,R)$, we have
    $$x \,\func{$\smile_h$}\, \func{$0_h$} \,\func{$\equiv$}\, \func{$0_h$} \,\func{$\smile_h$}\, x \,\func{$\equiv$}\, \func{$0_h$}$$
  \item The element $1_h : H^0(X,R)$ defined by
    \[
    \func{$1_h$} = \,\truncElem{\lambda x \to \func{$1_r$}}
    \]
    is a multiplicative unit. So, for $x : H^n(X,R)$, we have
    $$x \,\func{$\smile_h$}\, \func{$1_h$} \,\func{$\equiv$}\, \func{$1_h$} \,\func{$\smile_h$}\, x \,\func{$\equiv$}\, x$$
  \item The cup product is associative.
  \item The cup product left- and right-distributes over $\func{$+_h$}$.
  \item The cup product is graded-commutative. For $x : H^n(X,R)$ and $y : H^m(X,R)$, we have
    \[
    x \,\func{$\smile_h$}\, y \,\func{$\equiv$}\, (-\anum{1})^{nm}(y \,\func{$\smile_h$}\, x)
    \]
    In particular, when $R = \ZMod$ (i.e. $\Z/2\Z$), we have
    commutativity, since in this case inversion is given by the
    identity map.\footnote{ Graded commutativity for arbitrary
      coefficients was first formalized by~\citet{Baumann18} and is
      currently being developed in the \agdaCubical library. }
  \end{itemize}
  We note that the above identities, in most cases, hold up to a
  dependent path. For instance, right-annihilation holds up to a path
  in $\N$ identifying $n\,\func{$+$}\,\anum{0}$ with $n$.

\section{Formalizing graded rings}
\label{sec:formalizinggradedrings}

Both cohomology rings and polynomials rings are \emph{graded}
rings. This means that their underlying additive groups can be
decomposed as a direct sum $\bigoplus_{i : I} G_i$ of abelian groups
$G_i$ with a graded multiplicative operation
$\star : G_i \to G_j
\to G_{i + j}$. For $\star$ to be well-defined, we need the indexing
set $I$ to be a monoid $(I,e,+)$; for both polynomials and cohomology
rings, we pick
$(\N,\anum{0},\func{\_+\_})$. Indeed, the polynomial ring $R[X]$ is
trivially graded by $\N$ if we let $G_i$ constantly be $R$ and $\star$
the multiplication of $R$. The cohomology ring $H^*(X,R)$, on the
other hand, is defined as $\bigoplus_{i : \N} H^i(X,R)$ together with
the cup product.

In this section we study how these notions can be formalized
constructively in HoTT while avoiding decidability assumptions. The
reason not to restrict to types with decidable equality is to support
$\R[X]$ and $\mathbb{C}[X]$ as well as spaces for which the cohomology
groups do not have decidable equality. Furthermore, as we want to be
able to do concrete computations, we are careful to ensure that
everything is constructive. To achieve this, we crucially rely on
higher inductive types to define well-behaved quotient types.

\subsection{Direct sums}
\label{subsec:directsums}

To formalize graded rings we first need to formalize the direct sum of
a family of abelian groups over an arbitrary (not necessarily
monoidal) indexing set $I$. Given abelian groups $G_i$ indexed by a
set $I$, the direct sum is often defined as
\[
  \bigoplus_{i : I} G_i :=
  \{ (g_i)_{i \in I} \mid
    \exists \text{ finite } J \subset I, \forall n \notin J,~~g_n = 0 \in G_n \}
\]
This definition can be adapted to HoTT by taking dependent functions
$g : (i : I) \to G~i$ for which there \emph{merely} exists a finite
subset $J$ of $I$ such that for all $j \notin J$ we have $g~j ≡ 0$. To
make this precise we need a notion of finite subsets together with a
membership relation. Various approaches to finite sets in HoTT were
studied by Frumin,~Geuvers,~Gondelman and van der
Weide~\cite{FruminGeuvers+18}, but as is common in constructive
mathematics, the notions bifurcate into constructively distinct, but
classically equivalent, notions. For our purposes, the most suitable
notion seems to be Kuratowski finite sets where finite subsets can be
represented as duplicate-free lists of elements \citep[Theorem~2.8]{FruminGeuvers+18}. However, as remarked
above, for cohomology and polynomials rings we do not need general
direct sums, but only those indexed by $\N$. In this case, the
definition can be simplified as
\[
\bigoplus_{i : I} G_i := \{ (g_i)_{i \in \N} \mid \exists_{k \in \N} \forall_{n > k},~~g_n = 0 \in G_n \}
\]
This simplification is possible as $\N$ has a total order with
finite initial segments. Indeed, initial segments are finite subsets
and all finite subsets are included in an initial segment by computing
the maximum of the subset.
This definition is easy to formalize in \CubicalAgda:
\ExecuteMetaData[agda/latex/Section3.tex]{oplusfun}
Here $\func{⟨}\,G\,n\func{⟩}$ is the underlying type and \func{0⟨}
\var{G} \var{n} \func{⟩} is the zero of $G\,n$. To prove that this is
an abelian group we lift the operations pointwise, e.g. addition
can be defined as:
\ExecuteMetaData[agda/latex/Section3.tex]{oplusfunadd}
We omit the proof of \func{prf}, but it is easily proved using the
fact that its type is a proposition which means that the eliminator of
propositional truncation lets us extract the points at which \var{f}
and \var{g} become zero from \var{pf} and \var{pg}.

As the second component in \func{⊕Fun} is propositionally truncated, it
suffices to prove that the underlying functions are equal in order to show that
two elements of \func{⊕Fun} are equal:
\ExecuteMetaData[agda/latex/Section3.tex]{oplusfuneq}
As the operations are defined pointwise, the fact that \func{⊕Fun} is
an abelian group is then easily proved using function
extensionality. For example, commutativity is proved as follows:
\ExecuteMetaData[agda/latex/Section3.tex]{oplusfunaddcomm}

While \func{⊕Fun} is convenient for proving the group laws, it is not
very convenient for the multiplicative structure of the graded rings,
as we will see in the next section. To remedy this, we have also
formalized an equivalent definition using a HIT. This definition
works just as well for arbitrary indexing types $I$, so we consider
the following general form:
\ExecuteMetaData[agda/latex/Section3.tex]{oplushit}
In \func{⊕HIT}, the direct sum is generated by homogeneous elements
instead of being formed as a sum over elements in all degrees. The
\con{0⊕} constructor is the neutral element of \func{⊕HIT}, \con{base}
forms a homogeneous element of degree \var{n}, and \con{\_+⊕\_} lets
us take the sum of \func{⊕HIT} elements. The first three path
constructors ensure that the resulting type is a commutative monoid,
while the next two ensure that \con{0⊕} and \con{\_+⊕\_} reflect the
\func{0} and \func{\_+\_} of homogeneous elements. Finally, the last
constructor is necessary in order to ensure that the type is a set.
Without it, we could easily form non-trivial paths; for example,
\con{+⊕Comm} \var{x} \var{x}
would be a path from $x\,\con{+⊕}\,x$ to itself which would not necessarily be
equal to \func{refl}.

It is easy to give an abelian group structure to \func{⊕HIT} by
defining an inverse by pattern-matching:
\ExecuteMetaData[agda/latex/Section3.tex]{oplusneg}
We omit the last three cases of the definition as they are a little
bit more complicated. Essentially, the definition of \func{-⊕}
proceeds directly by structural recursion, negating the homogeneous
elements. It would be possible to add a constructor for \func{-⊕} to
the HIT, however this would have forced us to handle even more cases when
defining functions out of \func{⊕HIT}. Moreover, as \func{-⊕} is a
function we get many useful definitional equalities for concrete
inputs.

Having defined negation, we can easily prove that \func{⊕HIT} is an
abelian group for a general index type \var{I}. In order to avoid long
definitions with many cases, we prove a special eliminator which only
requires us to give the cases for point constructors when proving
propositions. As \func{⊕HIT} is a set, all of the equations between
\func{⊕HIT} elements are propositions, so this special eliminator
shortens the proofs substantially and we can avoid a lot of boilerplate
code. Defining special eliminators like this is a common pattern for
working with set truncated HITs in the \agdaCubical library.

\begin{remark}
  We could have proved that \func{⊕HIT} \func{ℕ} and \func{⊕Fun} are
  equivalent types and that this equivalence preserves the raw group
  structure (\func{0}, \func{\_+\_}, \func{-\_}). This way, the SIP
  would have allowed us to get the proofs of the abelian group laws
  for \func{⊕HIT} \func{ℕ} for free. However, as the proofs of these
  laws are easy to do in the general case, using the special
  eliminator for proving propositions, we formalized if for
  \func{⊕HIT} \var{I} directly.
\end{remark}

\subsection{Graded rings}
\label{subsec:gradedrings}

Having formalized direct sums, we now need to equip them with a ring
structure induced by a graded multiplicative operation
$\star : G_i \to G_j \to G_{i + j}$ over any monoid $(I,e,+)$. This means that we should define
a product operation of type
\[
  \bigoplus_{i : I} G_i \to \bigoplus_{i : I} G_i \to \bigoplus_{i : I} G_i
\]
using $\star$. For this product to satisfy the ring laws, the $\star$
operation has to satisfy suitable graded versions of them, e.g. it has
to have a multiplicative unit in degree $0$, be associative,
distribute over sums of elements of the same degree, etc. We refer the
interested reader to the formalization for the exact formulation of
these laws.

The $\star$ operation over the monoid $(\N,0,+)$ can be lifted
pointwise to a product operation on \func{⊕Fun} by the following
definition:
\[
(f * g)(n) = \sum_{i = 0}^n \uparrow_i^n (f(i) \star g(n - i))
\]
Here $\uparrow_i^n$ is a transport from $G_{i + (n -i)}$ to
$G_n$. This transport is necessary for \Agda to see that this is a
well-defined sum of elements in $G_n$. It is straightforward to show
that if $f$ and $g$ are zero after $k$ respectively $k'$ steps, then
the product is zero after at most $k \cdot k'$ steps. So $f * g$ is
indeed a well-defined element of \func{⊕Fun}.

Using the graded ring laws for $\star$ one can then prove the ring
laws for $*$. However, doing this gets surprisingly complicated
already for associativity as we need to fill in the gaps in the
following chain of equalities:
\begin{align*}
    &(f * (g * h)) (n) \\
    & = \sum_{i = 0}^n \uparrow_i^n (f(i) \star (g * h)(n - i)) \\
    & = \sum_{i = 0}^n \uparrow_i^n \left(f(i) \star \left(\sum_{j = 0}^{n -i} \uparrow_j^{n-i} (g(j) \star h(n-i-j))\right)\right) \\
    & = \cdots \\
    & = \sum_{i = 0}^n \uparrow_i^n \left(\left(\sum_{j = 0}^i \uparrow_j^i (f(j) \star g(i-j))\right) \star h(n-i)\right) \\
    & = \sum_{i = 0}^n \uparrow_i^n ((f * g) (i) \star h(n - i)) \\
    & = ((f * g) * h) (n)
\end{align*}
This might not look too bad on paper, especially if one ignores the
transports, but when one starts formalizing it in \CubicalAgda, one
quickly finds oneself descending deeper and deeper into transport
hell. This is exactly where \func{⊕HIT} shines; by instead working
mainly with homogeneous elements, we can lift $\star$ (for any monoid
$I$) and avoid all transports in the definition by writing the
following:\footnote{In practice, we give the same definition using
  pre-defined recursors for \func{$\oplus$HIT} which helps convince
  \Agda's termination checker. The definition by pattern-matching is
  merely included for pedagogical reasons.}
\ExecuteMetaData[agda/latex/Section3.tex]{oplusmul}
The most interesting case of the definition is the one where both
arguments are \con{base} elements. This corresponds to taking the
product of homogeneous elements, which of course gives another
homogeneous element, defined using \func{\_⋆\_}.

This gives a construction of the product operation, but we still have
to prove that it satisfies the ring laws. These laws are propositions
as they are equations between elements of a set, so we can again use
the special eliminator for proving proposition. It hence suffices to
prove the laws for the point constructors: \con{0⊕}, \con{base}, and
\con{\_+⊕\_}. Generally, the \con{0⊕} and \con{\_+⊕\_} cases are
trivial, which leaves the cases for \con{base}. This means that
we are left to prove the laws for homogeneous elements, which directly
follows from the graded ring laws for \func{\_⋆\_}. In particular,
associativity follows directly.

\begin{remark}
  The \Lean Mathlib \cite{mathlib} also has a formalization of direct
  sums and graded rings. There, $\bigoplus_{i \in I} G_i$ is defined as
  finitely supported dependent functions, i.e. dependent functions
  that are nonzero only at a finite set of points, just like in the
  first definition above. In fact, there is even a formal proof that a
  definition of $*$, similar to the sum involving transports, is
  associative. This proof heavily relies on tactics and automation to
  handle the complexity induced by the transports.  However,
  \CubicalAgda has very limited support for automation and because
  \CubicalAgda is proof relevant, this approach would most likely be a lot more
  challenging than in the proof irrelevant setting of \Lean.
\end{remark}

\subsubsection{Transporting the ring structure using the SIP}
\label{subsubsec:sip}

Having established that \func{⊕HIT} can be turned into a ring given
\func{\_⋆\_}, it is natural to ask how this relates to the ring
structure on \func{⊕Fun} using $*$. If we were to go through the
hurdles of transport hell, we would be able to equip \func{⊕Fun} with a
ring structure and, with some more work, we can then prove that this is
isomorphic to \func{⊕HIT} \func{ℕ}. The SIP then implies that these
rings are path-equal which means that all properties proved for one of
them also holds for the other. However, as discussed in
\cref{subsection:SIP}, the SIP actually gives us even more. If we,
instead of proving the ring laws for $*$ by hand, would prove that
\func{⊕HIT} \func{ℕ} and \func{⊕Fun} are equivalent as types and that
this equivalence preserves the \emph{raw} ring structure
(\func{0},\func{1},\func{-},\func{+},\func{*}), then we could use the
SIP to transport also the ring laws to get an induced ring structure
on \func{⊕Fun}.

Proving that \func{⊕HIT} \func{ℕ} and \func{⊕Fun} are equivalent types
can be done quite easily. In one direction, we send \con{0⊕} and
\con{+⊕} to the zero and addition of \func{⊕Fun}, and \con{base}
\var{n} \var{x} to the function that is \var{x} at position \var{n}
and zero otherwise. For the other direction, we can map a function to a
nested \con{+⊕} sum of \con{base} elements. It is direct to prove that
these maps cancel and hence establish that the types are
equivalent. We can now show that the raw ring structures is preserved
by the equivalence.  This is straightforward for all cases except
multiplication, which as usual boils down to homogeneous elements. The
proofs are a bit repetitive, but do not present any special technical
or conceptual difficulties. The SIP then gives an induced ring
structure on \func{⊕Fun} which is path-equal to the one on \func{⊕HIT}
\func{ℕ}, and we hence get that the rings satisfy the same properties.
This can be seen as a practical example, in the spirit of
\CoqEAL~\citep{CohenDenesMortberg13}, of using the SIP to avoid long
and technical proofs by changing data representation to a type where
some proofs are easier.

\section{Polynomial rings}
\label{sec:polynomialrings}

Recall that $R[X]$ can be defined by letting the underlying abelian
group be $\bigoplus_{\N}R$ and the multiplication be induced by the one
of $R$. Having formalized two constructions of $\bigoplus_\N$, we hence
get two equivalent representations of $R[X]$ which we will now
compare in detail. We will also compare them with a list based
representation, which clarifies the role of decidable equality in the
formalization, and discuss how $R[X_1,\dots,X_n]$ can be directly
encoded using \func{⊕HIT}.

\subsection{Polynomials as graded rings}

As discussed in the previous section, we can easily transport
properties back and forth between the \func{⊕Fun} and \func{⊕HIT}
representations of $R[X]$. In practice, however, the two
representations are quite different to work with.
As we have seen, the one based on \func{⊕Fun} is especially well-suited
for operations that can be defined pointwise, e.g.  addition, but
for operations that are not pointwise, things can get quite involved, because
they may depend on the $k$ after which the function becomes $0$.  For
example, in order to evaluate a polynomial at a point, one needs to take a sum
up to $k$; as we only merely have access to $k$, this gets somewhat
complicated. The \func{⊕HIT} definition has the benefit that it is
inductive, which means that operations can be defined by recursion and
properties proved by structural induction.  This means that some
definitions, like addition, get a bit longer than for \func{⊕Fun}, as
we need to write more cases. On the other hand, polynomial evaluation is much easier
to implement.

Also, from a computational perspective, the two representations
behave differently. The \func{⊕Fun} definition can be seen as a shallow
embedding, since polynomials are encoded as \Agda functions, which
makes them harder both to input and print. It is also complicated to
predict how efficient they are to compute with and how much memory
they use, as this depends on the internal representation of \Agda functions.

The HIT definition, on the other hand, is a deep embedding, which is easy to write and print. Furthermore, it is a sparse
representation where polynomials like $2X^3 + X^{100}$ are generated
by monomials and can be written compactly as
\[
  \con{base}\,\anum{3}\,\anum{2}\,\con{+⊕}\,\con{base}\,\anum{100}\,\anum{1}
\]
This resembles the polynomials on sparse Horner normal form of
\citet[Section~3]{GregoireMahboubi05}. In this representation, a
polynomial is either a constant $a \in R$ or on the form $a + pX^{n}$
for $a \in R$ and $p$ a polynomial in sparse Horner normal form. The
above polynomial would hence be encoded as $0 + (2 + X^{97})X^3$,
making it similar in terms of memory usage to \func{⊕HIT} \func{ℕ}. It
would be possible to implement this representation as a HIT where
equal representations of the same polynomial get identified, and then
prove that it is equivalent to \func{⊕HIT} \func{ℕ}.
However, a drawback of these sparse representations is that it
is easy to get complicated terms when working with concrete
polynomials. For example,
\[
  \con{base}\,\anum{3}\,\anum{2}\,\con{+⊕}\,\con{base}\,\anum{3}\,\anum{3}
\]
represents the same polynomial as \con{base} \anum{3} \anum{5} up to a
path. To remedy this, one should be careful to normalize polynomials
appropriately when doing efficient computations
\citep[Section~3.2]{GregoireMahboubi05}, which in turn might require
some decidability assumptions on \var{R} (e.g. to simplify \con{base} \var{n} \func{0r} to \con{0⊕}).

\subsection{List based polynomials and decidable equality}

Another common representation of polynomials is as a list of
coefficients; for instance, $[1,0,2,5]$ represents $1 + 2X^2 + 5X^3$. However,
one has to be careful to avoid zeroes at the end of the list, as for
example $[1,0,2,5,0,0]$ encodes the same polynomial as the list above.
This can be easily fixed by quotienting by a relation which equates
lists modulo trailing zeroes. The following HIT encodes this
very compactly:
\ExecuteMetaData[agda/latex/Section4.tex]{listpoly}
Proving \var{xs} \func{≡} \var{xs} \func{++} \func{replicate} \var{n}
\func{0r} for \var{xs} : \func{ListPoly} and \var{n} : \func{ℕ} is
easy, so this indeed equates a list with itself appended by
arbitrarily many zeroes. Furthermore, by a rather ingenious proof,
one can show that \func{ListPoly} is equivalent to \func{⊕Fun}, which
justifies omitting the set truncation in the definition; thus, this is
an alternative definition of $R[X]$. However, it is not well-suited
for sparse polynomials, as e.g. $2X^3 + X^{100}$ is encoded by a list with
$101$ elements, out of which $99$ are $0$.
Nevertheless, for dense polynomials it is not so bad and we avoid
having to normalize monomials for efficiency.

Something remarkable with all three of these representations is that we
avoid the assumption that $R$ is discrete (i.e. that its underlying
type has decidable equality), while still obtaining constructive
proofs that $R[X]$ is a commutative ring. This crucially relies on
HITs for constructing quotient types. However, as is common in
constructive algebra \citep{Mines}, we cannot compute neither the
degree nor the leading coefficient of these HIT polynomials. This is
especially clear for \func{ListPoly}, where we cannot know how many
trailing zeroes a list has without being able to test if an element is
zero. However, this is not as big a problem as one might
expect. Indeed, as both the \func{ListPoly} and \func{⊕HIT} are HITs, we
can construct functions on them by structural recursion and reason
about them by induction without having to talk about degrees.

For some applications it might be necessary to compute the degree or
leading coefficient though (e.g. for polynomial pseudo division
\citep{Knuth}). To this end, it is useful to restrict attention to
discrete rings $R$ and define $R[X]$ as lists with a nonzero last
coefficient, which makes the degree and leading coefficient
computable. This is the approach of the MathComp library in \Coq which
was used by \citet{Gonthier+13} to formalize the Feit-Thompson odd
order theorem and many other impressive results in algebra. This
representation can be seen as a normal form for the other
representations considered here. Indeed, when $R$ is discrete, it is
easy to reduce \func{ListPoly} to this form by dropping trailing
zeroes. By the equivalence of \func{ListPoly} with both \func{⊕HIT}
\func{ℕ} and \func{⊕Fun}, these can also be normalized to this form.

However, a caveat with representing polynomials on this normalized
form is that the invariant of having a nonzero last element has to be
preserved by all operations, complicating for example addition, since
trailing zeroes might have to be removed after adding the lists. One
also has to be careful not to have \Agda normalize the proofs that
the last element is nonzero. This also applies to \func{⊕Fun}, which is
also represented as a subtype. However, it does not apply to
\func{⊕HIT} and \func{ListPoly}, as these are HITs where we
instead work directly with representatives of equivalence classes of
polynomials.

\subsection{Multivariate polynomials}

Having discussed these representations of univariate polynomials, it is
natural to also consider multivariate polynomial rings
$R[X_1,\dots,X_n]$. These will be used in \cref{sec:cohomologyrings} where we
will prove that some cohomology rings are equivalent to multivariate
polynomial rings quotiented by ideals.

Often in algebra, $R[X_1,\dots,X_n]$ is represented as iterated
univariate polynomials $(((R[X_1])[X_2])\cdots)[X_n]$. This has the
benefit that properties can be proved for $R[X_1,\dots,X_n]$ by
proving that if the property holds for $R$, then it also holds for
$R[X]$, and then applying this $n$ times.
While this is convenient for proving properties of
$R[X_1,\dots,X_n]$, it is not always the best for doing computations.

To remedy this, we can let the indexing set be
$\Vect{\var{n}}{\func{ℕ}}$ in \func{⊕HIT} and obtain a direct
definition of $R[X_1,\dots,X_n]$. This definition still gives a simple
representation of polynomials where the \con{base} constructor
represents monomials; for instance, $\base{(4,0,3)}{\var{a}}$
corresponds to $aX_1^4X_3^3 \in R[X_1,X_2,X_3]$.
Having a direct definition of $R[X_1,\dots,X_n]$ is good for
computations as operations like multiplication become simpler if they
are directly defined instead of obtained by iterating the algorithm
for univariate polynomials. It also helps with proofs, e.g. proving
formally that
\[
  \Pco{R}{X_1,\dots,X_n}{X_1,\dots,X_n} \cong R
\]
using iterated univariate polynomials would be very cumbersome, but
with the \func{Vec} based representation it is easily proved by
\func{⊕HIT} induction.

For this reason, we are going to use \func{⊕HIT}
($\Vect{\var{n}}{\func{ℕ}}$) as our representation of
$R[X_1,\dots,X_n]$ in the next section. To avoid confusion when
working with two notions defined using \func{⊕HIT} (i.e. polynomials
and cohomology rings), we will from now on simply write e.g. $aX^2Y^3$
for $\base{(2,3)}{\var{a}}$ and use $0$ for \con{0⊕} as well as $+$ for \con{+⊕}
when working with $R[X_1,\dots,X_n]$.

\section{Cohomology rings}
\label{sec:cohomologyrings}

We are now finally ready to formalize cohomology rings. Recall, we
define the cohomology ring of $X$ with coefficients in a ring $R$ by
$H^*(X,R) = \bigoplus_{i:\mathbb{N}} H^i(X,R)$. In the formalization,
we use \func{⊕HIT} to formalize $H^*(X,R)$; the fact that it is a
ring follows immediately by construction, using the graded ring laws
of $\func{$\smile_h$}$.

Note that it is crucial that our definition of $\bigoplus$ avoided
assumptions about decidable equality --- it is not the case that
$H^i(X,R)$ has decidable equality for any $i$, $X$ and $R$. For a
simple counterexample, pick a ring $R$
without decidable equality (e.g. $R = \mathbb{R}$) and
consider the cohomology group $H^n(\mathbb{S}^n,R)$. As it is
isomorphic (as an abelian group) to $R$, it cannot have decidable
equality, thereby preventing us from even defining
$H^*(\mathbb{S}^n,R)$.

We will now discuss the various formalized computations of $H^*(X,R)$
for different spaces $X$ and rings $R$ which we have carried out.
We have formalized these computations as equivalences of \emph{rings},
but we make explicit in the proofs how they respect the graded
structure.
Polynomial rings are particularly useful in the computation of
these cohomology rings. While not all cohomology rings are isomorphic
to $R[X_1,\dots,X_n]/I$ for some commutative ring $R$ and ideal $I$,
there are many examples where a description using polynomial rings is
arguably the most tractable choice. These various computations will
allow us to conclude that various spaces are not equivalent, despite
having the same cohomology groups.

\subsection{Spheres}

As a warm-up, let us consider the integral cohomology ring of \func{$\mathbb{S}^1$}.
\begin{proposition}
  \label{prop:cohom-S1}
  We have $H^*(\func{$\mathbb{S}^1$},\Z) \cong \Z[X]/(X^2)$. 
\end{proposition}
\begin{proof}
  The cohomology groups of \func{$\mathbb{S}^1$} are easily proved to be:
  \[H^n(\func{$\mathbb{S}^1$}, \Z) \cong \begin{cases}
    \Z & \text{ if $n\leq \anum{1}$}\\
    0 & \text{ otherwise}
    \end{cases}
  \]
  For a synthetic proof of this, see e.g.~\citet[Section 5]{BLM}.
  For $i \leq \anum{1}$, let
  $\phi_i : \Z \cong H^i(\func{$\mathbb{S}^1$},\Z)$ be the above isomorphisms. We
  define a map $\psi : \Z[X] \to H^*(\func{$\mathbb{S}^1$},\Z)$ by:
  \begin{align*}
    \psi(0) &= \con{0\oplus} \\
    \psi(cX^0) &= \con{base}\,\anum{0}\,(\phi_0(c)) \\
    \psi(cX^1) &= \con{base}\,\anum{1}\,(\phi_1(c)) \\
    \psi(cX^{2+n}) &= \con{0\oplus} \\
    \psi(p + q)  &= \psi(p) \,\con{+\oplus} \,\psi(q) \\
    \vdots
  \end{align*}
  The remaining cases are immediate by construction as we are defining
  a map from \func{⊕HIT} to itself, and since $\phi_i$ is a family of
  group homomorphisms.  Proving that this map is a ring homomorphism
  is easy, since one only has to consider the cup product in low
  dimensions, in which case it simply corresponds to left- or
  right-multiplication over $\Z$. Moreover, as $\psi$ cancels on $X^2$,
  this allows us to view $\psi$ as a map
  $\Z[X]/(X^2) \to H^*(\func{$\mathbb{S}^1$},\Z)$. Its inverse
  $\psi^{-1} : H^*(\func{$\mathbb{S}^1$},\Z) \to \Z[X]/(X^2)$ is given
  (on homogeneous elements) by:
  \begin{align*}
    \psi^{-1}(\con{0\oplus}) &= 0 \\
    \psi^{-1}(\con{base}\,\anum{0}\,c) &= (\phi^{-1}_0(c))X^0 \\
    \psi^{-1}(\con{base}\,\anum{1}\,c) &= (\phi^{-1}_1(c))X^1 \\
    \psi^{-1}(\con{base}\,(\anum{2} + n)\,c) &= 0 \\
    \psi^{-1}(x \,\con{+\oplus} \,y)  &= \psi^{-1}(x) + \psi^{-1}(y)
  \end{align*}
  That $\psi$ and $\psi^{-1}$ cancel follows by construction.
\end{proof}
The above proof also works for higher spheres. We define these by
$(n-1)$-fold suspension of \func{$\mathbb{S}^1$}, where the suspension of a
type is defined in \CubicalAgda by
\ExecuteMetaData[agda/latex/Section5.tex]{Susp}
Let us explicitly define three recurring special cases:
\begin{mathpar}
  \func{$\mathbb{S}^2$} = \func{$\Sigma$}\,\func{$\mathbb{S}^1$} \and
  \func{$\mathbb{S}^3$} = \func{$\Sigma$}\,\func{$\mathbb{S}^2$} \and
  \func{$\mathbb{S}^4$} = \func{$\Sigma$}\,\func{$\mathbb{S}^3$}
\end{mathpar}
As our cohomology theory satisfies the Eilenberg-Steenrod axioms\footnote{See e.g.~\citet{CavalloMsc15} for the definition. The fact that our theory satisfies the axioms was proved by~\citet{BLM}.}, it follows from the
suspension axiom that, for $n \geq \anum{1}$, we have
\[
H^m(\func{$\mathbb{S}$}^n,\Z) \cong \begin{cases}
  \Z & m = \anum{0}, n = m \\
  0 & \text{otherwise}
  \end{cases}
\]
Hence, by a similar proof to that of~\cref{prop:cohom-S1}, we may conclude the following:
\begin{proposition}
  For $n \geq \anum{1}$, we have $H^*(\func{$\mathbb{S}$}^n,\mathbb{Z})\cong \mathbb{Z}[X]/(X^2)$.
\end{proposition}
\noindent We remark that the exact same method may be used to compute the cohomology rings of spheres with arbitrary coefficients, although this has not yet been formalized.

\subsection{The complex projective plane}

For a slightly more technical example, let us consider
\func{$\mathbb{C}P^2$}. We define it in terms of its usual CW decomposition
(see e.g.~\citet[Example 0.6]{Hatcher2002}), using the following homotopy pushout:
\[\begin{tikzcd}
	{\func{$\mathbb{S}$}^{\func{3}}} & \{*\} \\
	{\func{$\mathbb{S}$}^{\func{2}}} & {\func{$\mathbb{C}P^2$}}
	\arrow["\mathsf{\func{hopf}}"', from=1-1, to=2-1]
	\arrow[from=2-1, to=2-2]
	\arrow[from=1-1, to=1-2]
	\arrow[from=1-2, to=2-2]
	\arrow["\lrcorner"{anchor=center, pos=0.125, rotate=180}, draw=none, from=2-2, to=1-1]
\end{tikzcd}\]
Above, $\func{hopf} : \func{$\mathbb{S}$}^{\func{3}} \to
\func{$\mathbb{S}$}^{\func{2}}$ is the generator of
$\pi_3(\func{$\mathbb{S}$}^{\func{2}})$. In
\CubicalAgda, we may define \func{$\mathbb{C}P^2$} as the following HIT:
\ExecuteMetaData[agda/latex/Section5.tex]{CP2}
For the construction of $\func{hopf}$, see \citet[Chapter 2.5]{Brunerie16}. The
cup product on \func{$\mathbb{C}P^2$} was shown by \citet[Proposition~6.3.2]{Brunerie16}, for
the first time in HoTT/UF, to induce the usual multiplication on $\Z$ via
\begin{align*}
  \Z\times \Z \cong H^2(\func{$\mathbb{C}P^2$}) \times H^2(\func{$\mathbb{C}P^2$})
  \xrightarrow{\func{\textnormal{$\smile_h$}}} H^4(\func{$\mathbb{C}P^2$}) \cong \Z
\end{align*}
This was later also formalized by~\citet{BLM}. The ring structure on
$H^*(\func{$\mathbb{C}P^2$},\Z)$ is particularly interesting, since it plays a
fundamental role in Brunerie's synthetic proof of
$\pi_4(\func{$\mathbb{S}$}^{\func{3}}) \cong \Z_2$. We can now package this up into a
more compact statement.
\begin{proposition}\label{CP2}
  $H^*(\func{$\mathbb{C}P^2$}, \Z) \cong \Z[X]/(X^3)$
\end{proposition}

\begin{proof}
  The groups $H^n(\func{$\mathbb{C}P^2$},\Z)$ are described in \cref{table:cp2}.
  \begin{table}[h!]
    \centering
    \begin{tabular}{ |c|c|c| }
      \hline \textbf{Dimension} & \textbf{Isomorphic to} &
     \textbf{Generator} \\ \hline $H^0$ & $\Z$ & $\eta$
     \\ $H^2$ & $\Z$ & $\alpha$ \\ $H^4$ & $\Z$ &
     $\beta$ \\ $H^{n}, n \neq 0,2,4$ & $\mathbb{0}$ & $0$ \\ \hline
    \end{tabular}
    \caption{Cohomology groups of \func{$\mathbb{C}P^2$}.}
    \label{table:cp2}
  \end{table}

  Let us denote by $\phi_i$ the isomorphisms
  $\Z \cong H^i(\func{$\mathbb{C}P^2$},\Z)$, for $i \in \{0,2,4\}$. We
  define $\psi : \Z[X] \to H^*(\func{$\mathbb{C}P^2$},\Z)$ by:
\begin{align*}
  \psi(0) &= \con{0\oplus} \\
  \psi(cX^0) &= \con{base}\,\anum{0}\,(\phi_0(c)) \\
  \psi(cX^1) &= \con{base}\,\anum{2}\,(\phi_2(c))\\
  \psi(cX^2) &= \con{base} \,\anum{4}\,(\phi_4(c))\\
  \psi(cX^{3+n}) &= \con{0\oplus} \\
  \psi(p + q) &= \psi(p) \,\con{+\oplus} \,\psi(q)
  \\ \vdots
\end{align*}
The left out cases are automatic since the above, by definition, is a
homomorphism of abelian groups. Proving that the map is a ring
homomorphism boils down to showing one non-trivial identity:
\begin{align*}
  \psi(c_1X^1 \cdot c_2X^1) \,\func{$\equiv$}\, \psi(c_1X^1) \,\func{$\smile$}\, \psi(c_2X^1)
\end{align*}
This follows by $\alpha \,\func{$\smile_h$}\, \alpha
\,\func{$\equiv$}\, \beta$, which can be shown by an application of
the \emph{Gysin sequence} (see~\cite[Chapter~6]{Brunerie16}). We have:
\begin{align*}
  \psi(c_1X^1 \cdot c_2X^1) \,&\func{$\equiv$}\, \con{base} \,\anum{4}\,(\phi_4(c_1c_2)) \\
  &\func{$\equiv$}\, c_1c_2 \cdot \con{base} \,\anum{4}\,(\phi_4(\anum{1})) \\
  &\func{$\equiv$}\, c_1c_2 \cdot \con{base} \,\anum{4}\,\beta \\
  &\func{$\equiv$}\, c_1c_2 \cdot \con{base} \,\anum{4}\,(\alpha \,\func{$\smile_h$}\, \alpha) \\
  &\func{$\equiv$}\, c_1c_2 \cdot (\con{base} \,\anum{2}\,\alpha \,\func{$\smile$}\, \con{base} \,\anum{2}\,\alpha) \\
  &\func{$\equiv$}\, c_1c_2 \cdot (\con{base} \,\anum{2}\,(\phi_2\,\anum{1}) \,\func{$\smile$}\, \con{base} \,\anum{2}\,(\phi_2\,\anum{1})) \\
  &\func{$\equiv$}\, \con{base} \,\anum{2}\,(\phi_2(c_1)) \,\func{$\smile$}\, \con{base} \,\anum{2}\,(\phi_2(c_2)) \\
  &\func{$\equiv$}\, \psi(c_1X^1) \,\func{$\smile$} \, \psi(c_2X^1)
\end{align*}
Moreover, as $\psi(X^3) = \con{0\oplus}$, we may view $\psi$ as a homomorphism
$\Z[X]/(X^3) \to H^*(\func{$\mathbb{C}P^2$},\Z)$. For its inverse, one proceeds
just like in~\cref{prop:cohom-S1}, and proving that the maps cancel
is equally straightforward.
\end{proof}

Cohomology rings are particularly useful for distinguishing ``similar''
spaces. Compare, for instance, \func{$\mathbb{C}P^2$} with the wedge sum
$\SW$, the space obtained by gluing
together the base points of \func{$\mathbb{S}^{\func{2}}$} and
\func{$\mathbb{S}^{\func{4}}$}. In~\CubicalAgda, we may define the wedge sum of two
pointed types by
\ExecuteMetaData[agda/latex/Section5.tex]{Wedge}

The cohomology groups of \func{$\mathbb{C}P^2$} and $\SW$ are the same in each dimension. There is an intuitive
reason for this: $\SW$ has one connected
component, one $3$-dimensional hole, one $5$-dimensional hole, and
nothing more. Hence, the cohomology groups should coincide with those
of \func{$\mathbb{C}P^2$}. However, the two spaces are not homotopy equivalent.
One way to see this is to define a predicate
$P : \mathbb{N}^2 \times \func{Type} \to \func{hProp}$ such that
$P(n,m,X)$ expresses that the cup product
$\func{$\smile_h$}\, : H^n(X,\Z) \times H^{m}(X,\Z) \to H^{n+m}(X,\Z)$ is
trivial. For our two spaces, we can show that
$P(\anum{2},\anum{2},\SW)$ holds, while
$P(\anum{2},\anum{2},\func{$\mathbb{C}P^2$})$ does not. This method was used by~\citet{BLM}
to distinguish the torus from
$\func{$\mathbb{S}$}^{\func{2}}\,\func{∨}\,\func{$\mathbb{S}^1$}\,\func{∨}\,\func{$\mathbb{S}^1$}$. This approach is,
however, rather ad hoc, as the predicate $P$ will have to be adjusted
if we want to capture other aspects of \func{$\smile_h$} than
triviality. With cohomology rings, we may instead simply conclude that
$\func{$\mathbb{C}P^2$}\, \not\simeq \, \SW$ from the
fact that
$H^*(\func{$\mathbb{C}P^2$},\Z) \,\not \cong \, H^*(\SW,\Z)$.

\begin{proposition}\label{S2wedgeS1}
  $H^{*}(\SW,\Z) \,\func{$\cong$}\,
  \Z[X,Y]/(X^2,XY,Y^2)$
\end{proposition}
\begin{proof}
  Let $\phi_i$ again be the
  isomorphisms $\Z \cong H^n(\SW,\Z)$, for
  $i \in \{0,2,4\}$. We define $\psi : \Z[X,Y] \to
  H^{*}(\SW,\Z)$ by
  \begin{align*}
    \psi(0) &= \con{0\oplus} \\
    \psi(c) &= \con{base}\,\anum{0}\,(\phi_0(c)) \\
    \psi(cX) &= \con{base}\,\anum{2}\,(\phi_2(c)) \\
    \psi(cY) &= \con{base}\,\anum{4}\,(\phi_4(c)) \\
    \psi(cX^nY^m) &= \con{0\oplus} & \text{if }n+m \geq 2\\
    \psi(p + q)  &= \psi(p) \,\con{+\oplus} \,\psi(q)
  \end{align*}

  That $\psi$ induces an isomorphism $\Z[X,Y]/(X^2,XY,Y^2)
  \cong H^{*}(\SW,\Z)$ is equally
  straightforward to verify as in the proofs of \cref{prop:cohom-S1,CP2}.
\end{proof}

\noindent Combining \cref{CP2,S2wedgeS1} we get the following:

\begin{corollary}
  $\func{$\mathbb{C}P^2$}\, \not\simeq\, \SW$
\end{corollary}

\subsection{The Klein bottle and the real projective plane with an
  adjoined circle}

Sometimes integral cohomology rings are not enough to distinguish
spaces and we have to change the coefficient ring to something else
than $\Z$. For an example of this, let us introduce two new spaces: the
Klein bottle \func{K²} and the real projective plane \func{ℝP²}. In
\CubicalAgda, they are defined by:
\ExecuteMetaData[agda/latex/Section5.tex]{Klein}
\ExecuteMetaData[agda/latex/Section5.tex]{RP2}
The intuition behind
these HITs is that they precisely capture the usual folding
diagrams representing each space. \newline
\begin{minipage}[b]{4cm}
  \begin{center}
\[\begin{tikzcd}
	\bullet & \bullet \\
	\bullet & \bullet
	\arrow["{\con{line2}}"', from=2-1, to=2-2]
	\arrow["{\con{line2}}", from=1-1, to=1-2]
	\arrow["{\con{line1}}"', from=1-1, to=2-1]
	\arrow["{\con{line1}}"', from=2-2, to=1-2]
\end{tikzcd}\]
$\Klein$
\end{center}
\end{minipage}
\begin{minipage}[b]{4cm}
  \begin{center}
\[\begin{tikzcd}
	\bullet \\
	\bullet
	\arrow["{\con{line}}"', bend left=30, swap, from=2-1, to=1-1]
	\arrow["{\con{line}}"', bend left=30, swap, from=1-1, to=2-1]
\end{tikzcd}\]
\newline
$\RP$
  \end{center}
\end{minipage}
\newline

\text{Consider} $\Klein$ and $\RPW$. The two spaces both have integral cohomology groups
\[
H^n(\Klein,\mathbb{Z}) \cong H^n(\RPW,\mathbb{Z}) \cong \begin{cases}
  \Z, & n \leq 1 \\
  \Z_2 & n = 2 \\
  0 & n > 2
  \end{cases}
\]
One can show that $\func{$\smile_h$}\, : H^1(X,\Z) \times H^1(X,\Z) \to H^2(X,\Z)$
vanishes for both spaces. Thus, we can easily show the following:
\begin{proposition}
  We have the following isomorphisms:
  \[
    H^*(\Klein,\Z) \cong H^*(\RPW,\Z)\cong
    \Z[X,Y]/(X^2,XY,2Y,Y^2)
  \]
\end{proposition}
\begin{proof}
  Let $S$ be either of the two spaces in question. The proof is close
  to identical to that of \cref{CP2} -- we let $X$ correspond to the
  generator of $H^1(S,\Z)$ and $Y$ correspond to the generator of
  $H^2(S,\Z)$. The only difference is the $2$-torsion in $H^2(S,\Z)$
  for both spaces, which is accounted for by the factor of $2Y$ in the
  quotient.
\end{proof}
We can, however, distinguish $\Klein$ and $\RPW$ by computing their cohomology rings with $\Z_2$
coefficients. In this case, both spaces still have the same non-trivial cohomology
groups, as described in \cref{table:k2}.
\begin{table}[h!]
  \centering
  \begin{tabular}{|c|c|c|c|}
    \hline
    \textbf{Dim.} & \textbf{Space} & \textbf{Isomorphic to} & \textbf{Generators} \\
    \hline
    $H^0$ & \begin{tabular}{c} $\Klein$ \\ $\RPW$ \end{tabular} & $\Z_2$ & \begin{tabular}{c} $\eta_1$ \\ $\eta_2$ \end{tabular}\\
    \hline
    $H^1$ & \begin{tabular}{c} $\Klein$ \\ $\RPW$ \end{tabular} & $\Z_2 \times \Z_2$ & \begin{tabular}{c} $\alpha_1,\beta_1$ \\ $\alpha_2,\beta_2$ \end{tabular}\\
    \hline
    $H^2$ & \begin{tabular}{c} $\Klein$ \\ $\RPW$ \end{tabular} & $\Z_2$ & \begin{tabular}{c} $\gamma_1$ \\ $\gamma_2$ \end{tabular}\\
    \hline
  \end{tabular}
  \caption{Cohomology groups of $\Klein$ and $\RPW$.}
  \label{table:k2}
\end{table}

The characterization of the cup product on the rings is non-trivial
and will have to be covered in a separate paper. For $\Klein$ it is
described, on generators, by
\begin{align}\label{cupIds1}
\alpha_1^2 \,\,\func{$\equiv$}\,\, \alpha_1 \,\func{$\smile_h$}\, \beta_1 \,\,\func{$\equiv$}\,\, \beta_1 \,\func{$\smile_h$}\, \alpha_1 \,\,\func{$\equiv$}\,\, \gamma_1 \quad \textrm{and} \quad
\beta_1^2 \,\,\func{$\equiv$}\,\, 0.
\end{align}
On $\RPW$ it is described by
\begin{align}\label{cupIds2}
\alpha_2 \,\func{$\smile_h$}\, \beta_2 \,\,\func{$\equiv$}\,\, \beta_2 \,\func{$\smile_h$}\, \alpha_2 \,\,\func{$\equiv$}\,\, \beta_2^2 \,\,\func{$\equiv$}\,\, 0 \quad \textrm{and} \quad \alpha_2^2 \,\,\func{$\equiv$}\,\, \gamma_2.
\end{align}
These relations let us describe the two cohomology rings.
\begin{proposition}\label{theorem:Klein}
  $H^*(\Klein,\ZMod) \cong \ZMod[X,Y]/(X^3,Y^2,X^2+XY)$
\end{proposition}
\begin{proof}
  In order to define $\psi : \ZMod[X,Y] \to H^*(\Klein,\ZMod)$, let us, for the sake of conciseness, be somewhat more informal than before. It suffices to define it on generators, i.e. $1$, $X$ and $Y$:
  \begin{align*}
    \psi(1) &= \con{base}\,\anum{0}\,\eta_1 \\
    \psi(X) &= \con{base}\,\anum{1}\,\alpha_1 \\
    \psi(Y) &= \con{base}\,\anum{1}\,\beta_1
  \end{align*}
  This induces the rest of the definition as $\psi$ needs to be a group homomorphism.
  Moreover, by the previous equations $\psi$ is automatically a ring homomorphism.
  We observe that $\psi$ vanishes
  on $(X^3,Y^2,X^2+XY)$:
  \begin{align*}
    \psi(X^3) &= \con{base}\,\anum{3}\,(\alpha^3_1) = \con{base}\,\anum{3}\,\func{$0_h$} = \con{0\oplus} \\
    \psi(Y^2) &= \con{base}\,\anum{3}\,(\beta^2_1) = \con{base}\,\anum{2}\,\func{$0_h$} = \con{0\oplus} \\
    \psi(X^2 + XY) &= \psi(X^2) \,\con{+\oplus}\, \psi(XY) = \con{base}\,\anum{2}\,(\alpha_1^2\, \func{$+_h$} \,(\alpha_1 \,\func{$\smile_h$}\, \beta_1)) \\
    &= \con{base}\,\anum{2}\,(\alpha_1^2 \,\func{$+_h$}\, \alpha_1^2) = \con{base}\,\anum{2}\, \func{$0_h$} \\
    &= \con{0\oplus}
  \end{align*}
  The first of these identities follows from $H^3(\Klein,\ZMod)$ being trivial, while the other two follow from the relations in \eqref{cupIds1} and the fact that the coefficients are $\ZMod$.

Hence, we may view $\psi$ as a homomorphism $$\ZMod[X,Y]/(X^3,Y^2,X^2+XY) \to H^*(\Klein,\ZMod)$$ The fact that it is an isomorphism is a straightforward consequence of the identities described in \eqref{cupIds1}.
\end{proof}

\if{
\begin{proof}
  We define the map $\psi : \ZMod[X,Y] \to H^*(\Klein,\ZMod)$ on homogeneous elements by
  \begin{align*}
    \psi(0) &= \con{0\oplus} \\
    \psi(X^nY^m) &= \con{base}\,(n+m)\,(F(n) \smile G(m)) \\
    \psi(p + q)  &= \psi(p) \,\con{+\oplus} \,\psi(q)
  \end{align*}
  where
  $F,G : (n : \mathbb{N}) \to H^n(\Klein)$ are given by \newline
  \begin{minipage}[t]{4cm}
      \vspace{-.3cm}
    \begin{align*}
    F(0) &= \eta_1 \\
    F(1) &= \alpha_1 \\
    F(2) &= \gamma_1 \\
    F(3 + n) &= 0
  \end{align*}
\end{minipage}
  \begin{minipage}[t]{4cm}
    \vspace{-.3cm}
  \begin{align*}
    G(0) &= \eta_1 \\
    G(1) &= \beta_1 \\
    G(2 + n) &= 0
  \end{align*}
  \end{minipage}
  \vspace{.2cm}
\newline
\noindent Note that
  $$\gamma_1 + \alpha_1 \smile \beta_1 = 2 \gamma_1 = 0$$ Hence, the
map vanishes on the elements of $(X^3,Y^2,X^2+XY)$. Let us spell out
why the map is a ring homomorphism. First, note that $F$ and $G$ sends
$+$ to $\smile$; this is a consequence of the identities
\begin{align*}
  \alpha_1 \smile \alpha_1 &= \gamma_1 \\
  \beta_1 \smile \beta_1 &= 0
\end{align*}
In order to show that $\psi$ is a ring homomorphism, the crucial identity that we need to show is
\[
\psi(X^{n+s}Y^{m+t}) = \psi(X^nY^m) \smile \psi(X^sY^t)
\]
This boils down to proving that
\[
F(n+s) \smile G(m+t) = (F(n) \smile G(m)) \smile (F(s) \smile G(t))
\]

We are careful to note that the above equality only holds up to an
identification of $(n+s) + (m+t)$ with $(n + m) + (s + t)$. This makes
the formalization slightly more bureaucratic, but let us stay informal
here. We have
\begin{align*}
  F(n+s) \smile G(m+t) = (F(n) \smile F(s)) \smile (G(m) \smile G(t))
\end{align*}
Rearranging the right-hand side using associativity and commutativity
for $\ZMod$-cohomology, we are done. We thus get that $\psi$ may be
seen as a homomorphism $$\ZMod[X,Y]/(X^3,Y^2,X^2+XY) \to H^*(\Klein)$$ The
inverse is induced by the following identifications.
\begin{align*}
  \eta_1 &\mapsto 1\\
  \alpha_1 &\mapsto X \\
  \gamma_2 &\mapsto X^2
\end{align*}
Proving that these maps cancel is an easy lemma.
\end{proof}
}\fi

\begin{proposition}\label{theorem:RP2}
  We have the following isomorphism:
  \[
    H^*(\RPW,\ZMod)\cong
    \ZMod[X,Y]/(X^3,XY,Y^2)
  \]
\end{proposition}
\begin{proof}
  Similar to the proof of \cref{theorem:Klein}, using the identities described in \eqref{cupIds2}.
\end{proof}
And thus, since the two spaces in question have different cohomology
rings, we arrive at the main result of this section.
\begin{corollary}
  $\Klein \not\simeq \RPW$
\end{corollary}

\section{Conclusions and future work}
\label{sec:conclusions}

In this paper, we have developed
cohomology rings synthetically and made multiple computations of such rings.
To the best of our knowledge, no other systems has a formalization of
cohomology rings for some cohomology theory.
We have
relied heavily on features of HoTT/UF, including univalence and HITs,
to do this in a way that is fully constructive without sacrificing
generality. The formalization also relies on the cubical SIP of
\citet{ACMZ21} to conveniently transport proofs between equivalent
structured types when developing the theory about direct sums in order
to get the best out of both representations. The fact that this was
done in \CubicalAgda where the univalence axiom, and hence the SIP,
has computational content means that it can be seen as a form of
\emph{data refinement} where we change data representation to
facilitate easier proofs while not sacrificing computation.

This form of data refinement is similar to the one in \CoqEAL
\citep{CohenDenesMortberg13}
and in the \systemname{Isabelle/HOL} tool \systemname{Autoref}
\citep{Lammich13}. One difference in our work is that we not only
refine programs, but also proof terms. Furthermore, both \func{⊕Fun}
and \func{⊕HIT} are \emph{proof-oriented} types in that they are both
well-suited for doing different kinds of proofs, while in \CoqEAL the
refinement was done from proof-oriented to computation-oriented
types. Another interesting difference is that HITs can be both
proof-oriented and computation-oriented at the same time. For example,
it is not necessary to refine \func{ListPoly} to ordinary lists as
concrete computations will only involve the point
constructors. Exploring these observations further is planned for
future work.

In \cref{table:summary}, we have summarized the computations of
cohomology rings which we have formalized. The additions to the
library required for these computation amount to about $8800$ lines of
code. However, this number should be taken with a grain of salt, as
this project was not a standalone development. In addition to these
$8800$ lines of code, various already existing parts of the library had to be
extended and rewritten, so the actual number of lines written is
bigger.
\begin{table}[h!]
  \centering
  \begin{tabular}{|c|c|c|}
    \hline
    Space & Coefficients & Cohomology ring \\
    \hline \hline
    $\func{$\mathbb{S}$}^n$ & $\Z$ & $\Z[X]/(X^2)$ \\
    \hline
    \func{$\mathbb{C}P^2$} & $\Z$ & $\Z[X]/(X^3)$ \\
    \hline
    \func{$\mathbb{S}^2 \,\vee \,\mathbb{S}^4$} & $\Z$ & $\Z[X,Y]/(X^2,XY,Y^2)$ \\
    \hline
    \func{$K^2$} & \begin{tabular}{c} $\Z$ \\ $\ZMod$ \end{tabular} & \begin{tabular}{c} $\Z[X,Y]/(X^2,XY,2Y,Y^2)$ \\ $\ZMod[X,Y]/(X^3,Y^2, X^2 + XY)$ \end{tabular} \\
    \hline
    \func{$\mathbb{RP}^2 \,\vee \, \mathbb{S}^1$} & \begin{tabular}{c} $\Z$ \\ $\ZMod$ \end{tabular} & \begin{tabular}{c} $\Z[X,Y]/(X^2,XY,2Y,Y^2)$ \\ $\ZMod[X,Y]/(X^3,Y^2,XY)$ \end{tabular}\\
    \hline
  \end{tabular}
  \caption{Summary of cohomology rings}
  \label{table:summary}
\end{table}

At the moment, our computations of cohomology rings are somewhat
ad-hoc, as each ring isomorphism is given by directly defining a
homomorphism with an explicitly constructed inverse. It seems plausible that
some of the complexity could be handled more conveniently by instead
first defining $f : R[X_1,\dots,X_n] \to H^*(X,R)$ and then
establishing that it is surjective. This would
mean that $H^*(X,R) \cong R[X_1,\dots,X_n] /
\mathsf{ker}(f)$ and the problem would be reduced to computing the generators of $\mathsf{ker}(f)$. Classically, one would typically rely on $R$ being
Noetherian, which carries over to $R[X_1,\dots,X_n]$ by Hilbert's
basis theorem. This then implies that the kernel is finitely
generated. This does, however, not work very well constructively and we
would not be able to extract a list of generators of the
ideal from the proof. Instead, we would have to rely on the theory of coherent rings which has previously been formalized in \Coq by
\citet{coherent}, but in order to use it in this context we would need
a proof that $R[X_1,\dots,X_n]$ is coherent, which boils down to
Gröbner basis computations \citep{BuchbergerConstructive}. Such an
approach, while interesting, would be a very major undertaking.

A less challenging direction of future work would be to compute
cohomology rings by hand for some more spaces. For example, the
cohomology ring of the torus should be very feasible to compute using
our machinery. A much more ambitious project would then be to
generalize this to a Künneth formula for computing cohomology of
general products of spaces. Another class of spaces to consider is
that of higher-dimensional real/complex projective spaces. \citet{RPn} gave a HoTT
formalization of $\mathbb{R}P^n$ and $\mathbb{R}P^{\infty}$, so it
would be interesting to formalize also these spaces and compute their
cohomology groups and rings.

In \cite{BLM} \emph{reduced} cohomology groups were also
formalized. These too could be lifted to construct reduced cohomology
rings $\widetilde{H}^*(X,R)$. Such a formalization would probably
benefit from the SIP as we could transport operations and results from
$H^*(X,R)$ in order to obtain a compact construction of
$\widetilde{H}^*(X,R)$. This could then be used to give general
theorems for computing the reduced cohomology ring of wedge sums:
$$\widetilde{H}^*(X \vee Y,R) \cong \widetilde{H}^*(X,R) \times \widetilde{H}^*(Y,R)$$

\bibliographystyle{ACM-Reference-Format}
\bibliography{refs}

\end{document}